\documentclass[11pt]{amsart}

\usepackage[american]{babel}
\usepackage[babel, final]{microtype}
\usepackage[all]{xy}
\usepackage{ifpdf}
\usepackage{comment}
\usepackage{multirow}
\usepackage{tikz-cd}
\usepackage{soul}

\usepackage{import}

\usepackage{float}

\usepackage{subcaption}

\usepackage[pdftex,%
    includehead,%
    includefoot,%
    nomarginpar,%
    lmargin=1in,%
    rmargin=1in,%
    tmargin=1in,%
    bmargin=1in,%
]{geometry}
\usepackage{amsmath,amsthm,amssymb,amsfonts,mathrsfs,tikz,tikz-cd,bm,verbatim,mathtools,hyperref,caption,enumitem,csquotes,float}

\usepackage{hyperref}
\hypersetup{
    colorlinks=true,
    linkcolor=blue,
    citecolor=blue,
}

\makeatletter
\def\@tocline#1#2#3#4#5#6#7{\relax
  \ifnum #1>\c@tocdepth 
  \else
    \par \addpenalty\@secpenalty\addvspace{#2}%
    \begingroup \hyphenpenalty\@M
    \@ifempty{#4}{%
      \@tempdima\csname r@tocindent\number#1\endcsname\relax
    }{%
      \@tempdima#4\relax
    }%
    \parindent\z@ \leftskip#3\relax \advance\leftskip\@tempdima\relax
    \rightskip\@pnumwidth plus4em \parfillskip-\@pnumwidth
    #5\leavevmode\hskip-\@tempdima
      \ifcase #1
       \or\or \hskip 1em \or \hskip 2em \else \hskip 3em \fi%
      #6\nobreak\relax
    \dotfill\hbox to\@pnumwidth{\@tocpagenum{#7}}\par
    \nobreak
    \endgroup
  \fi}
\makeatother

\usepackage[hang,flushmargin]{footmisc}

\usepackage[backend=biber,style=alphabetic,sorting=nyt, url=false,doi=false,isbn=false,backref=true]{biblatex}
\DefineBibliographyStrings{english}{
  backrefpage={$\uparrow$},
  backrefpages={$\uparrow$}
}

\setlist[itemize]{noitemsep} 

\addto\extrasamerican{%
}

\newtheorem{theorem}{Theorem}[section]
\newtheorem{corollary}{Corollary}[section]
\newtheorem{proposition}{Proposition}[section]
\newtheorem{lemma}{Lemma}[section]
\theoremstyle{definition}
\newtheorem{definition}{Definition}[section]

\makeatletter
\let\c@conjecture=\c@theorem
\let\c@corollary=\c@theorem
\let\c@proposition=\c@theorem
\let\c@lemma=\c@theorem
\let\c@definition=\c@theorem
\let\c@example=\c@theorem
\let\c@remark=\c@theorem
\let\c@equation\c@theorem
\let\c@question\c@theorem
\makeatother

\def\makeautorefname#1#2{\expandafter\def\csname#1autorefname\endcsname{#2}}
\makeautorefname{proposition}{Proposition}
\makeautorefname{lemma}{Lemma}
\makeautorefname{definition}{Definition}
\makeautorefname{example}{Example}
\makeautorefname{question}{Question}
\makeautorefname{conjecture}{Conjecture}
\makeautorefname{section}{Section}
\makeautorefname{claim}{Claim}
\makeautorefname{equation}{Equation}
\makeautorefname{remark}{Remark}
\makeautorefname{corollary}{Corollary}

\newcommand{\ZZ}{\mathbb{Z}}

\newcommand{\RR}{\mathbb{R}}

\newcommand{\RP}{\mathbb{RP}}
\newcommand{\CP}{\mathbb{CP}}

\newcommand{\cS}{\mathcal{S}}
\newcommand{\cR}{\mathcal{R}}
\newcommand{\cB}{\mathcal{B}}
\newcommand{\cG}{\mathcal{G}}

\newcommand{\mgamma}{\mathcal{M}(\Gamma)}
\newcommand{\D}{\mathcal{D}}
\newcommand{\T}{\mathbb{T}}  

\newcommand{\x}{\mathbf{x}}
\newcommand{\y}{\mathbf{y}}
\newcommand{\xtilde}{\widetilde{\x}}
\newcommand{\ytilde}{\widetilde{\y}}

\DeclareMathOperator{\II}{I}

\DeclareMathOperator{\Ima}{Im}

\newcommand{\GG}{\mathbb{G}}

\newcommand{\cM}{\mathcal{M}}

\title{Moduli Spaces of Lagrangian Surfaces in $\CP^2$ obtained from Triple Grid Diagrams}

\author{Devashi Gulati}
\address{Department of Mathematics, University of Georgia, Athens, GA 30602}
\email{\href{gulati.devashi@gmail.com}{gulati.devashi@gmail.com}}

\author{Peter Lambert-Cole}
\address{Department of Mathematics, University of Georgia, Athens, GA 30602}
\email{\href{plc@uga.edu}{plc@uga.edu}}

\begin{document}

\maketitle

\begin{abstract}
Links in $S^3$ as well as Legendrian links in the standard tight contact structure on $S^3$ can be encoded by grid diagrams.  These consist of a collection of points on a toroidal grid, connected by vertical and horizontal edges.  Blackwell, Gay and second author studied triple grid diagrams, a generalization where the points are connected by vertical, horizontal and diagonal edges.  In certain cases, these determine Lagrangian surfaces in $\CP^2$.  However, it was difficult to construct explicit examples of triple grid diagrams, either by an approximation method or combinatorial search.  We give an elegant geometric construction that produces the moduli space of all triple grid diagrams.  By conditioning on the abstract graph underlying the triple grid diagram, as opposed to the grid size, the problem reduces to linear algebra and can be solved quickly in polynomial time. 
\end{abstract}

\section{Introduction}

From the introduction of bridge position for surfaces \cite{MZ-bridge_2017}, it was immediately clear that one could define and search for `triple grid diagrams' for surfaces in $\CP^2$ (with respect to the genus 1 trisection), just as there exist grid diagrams for knots and links in $S^3$ and lens spaces (with respect to the genus 1 Heegaard splitting).  A grid diagram for a link is a collection of points on the genus-1 Heegaard torus for $S^3$, connected by vertical and horizontal edges on the torus.  Moreover, rotating a grid diagram by $45^{\circ}$ produces a front projection of a Legendrian link in the standard tight contact structure on $S^3$.  Every smooth link and every Legendrian link can be described by some grid diagram.  A triple grid diagram is an analogous description of a surface in bridge position in $\CP^2$, consisting of points on the torus and three collections of edges, having vertical, horizontal and diagonal slope.  Each of the three pairs of edges produces a standard grid diagram of a Legendrian link in $S^3$, hence the term `triple' grid diagram.

However, it proved difficult to construct interesting triple grid diagrams beyond some basic families of examples, either by an approximation method or by a combinatorial search.  Adding a third slope condition on the points is highly constraining, therefore it is hard to approximate an arbitrary surface in bridge position by a triple grid diagram.  In a different direction, the points of a grid diagram on an $n \times n$-grid can be encoded by a pair of permutations in the symmetric group $S_n$.  One can then attempt to search through all pairs of permutations and select  pairs that satisfy a third slope condition.  This approach would take $O((n!)^2)$-time to produce all triple grid diagrams of grid size $n$, which is computationally intractable for large $n$.  For small $n$, there are few distinct examples up to symmetry.

In this paper, we give an elegant geometric construction that produces all triple grid diagrams.  The insight is to conduct the search by conditioning on the abstract, Tait-colored cubic graph $\Gamma$ one is trying to realize by a triple grid diagram, as oppposed to conditioning on the grid size $n$.  From this perspective, the problem is purely linear algebra and can be solved quickly in polynomial time.  In fact, we can prove the existence of triple grid diagrams purely by counting dimensions of intersecting linear subspaces.

\begin{theorem}\label{thm:main_thm}
Let $\Gamma$ be an abstract Tait-colored cubic graph with $2b$ vertices.  The moduli space $\mgamma$ of geometric triple grid diagrams with graph-type $\Gamma$ is a smooth manifold.

\begin{enumerate}
    \item If $b = 1$, then the moduli space is empty.
    
    \item If $b = 2$, then the moduli space is 2-dimensional and consists of $\T^2$-translates of a single element, the $\mathbb{RP}^{2}$-graph in Figure \ref{fig:b2_example}.

    \item If $b \geq 3$, then $\mgamma$ is nonempty and has dimension $d \geq b$.
\end{enumerate}

\end{theorem}

This theorem is a combination of Theorem \ref{thrm:moduli-dimension} and Examples \ref{ex:b1} and \ref{ex:b2}.

\subsection{Triple grid diagrams and Lagrangian surfaces}

One geometric motivation to study triple grid diagrams is that they determine Lagrangian surfaces in $\CP^2$.  A triple grid diagram $D$ determines three standard grid diagrams $\GG_{1},\GG_{2},\GG_{3}$.  Via a well-known method, grid diagrams determine Legendrian links in the standard tight contact structure $(S^3,\xi_{std})$.  

Start with $(\CP^2, \omega_{FS})$ and remove three Darboux balls.  The result is a symplectic 4-manifold $(X,\omega)$ with three boundary components $Y_1,Y_2,Y_3$, each with an inward-pointing Liouville vector field $\rho_i$.  The 1-form $\alpha_i = \omega(\rho_i,-)$ is a contact form for the standard tight contact structure $\xi_{std}$ on $S^3$.  This symplectic 4-manifold is a {\it strong symplectic cap} for the disjoint union of three copies of $(S^3,\xi_{std})$.  A {\it Lagrangian cap} is a surface $L$ that intersects the boundary transversely, such that $\omega|_L = 0$ and each Liouville vector field $\rho_i$ is tangent to $L$ in some neighborhood of the boundary.  Consequently, the link $\Lambda_i = L \pitchfork Y_i$ is a Legendrian link tangent to the contact structure $\xi_{std}$ along $Y_i$.

\begin{theorem}[\cite{BGLC_2023}]
\label{thrm:BGLC}
    A geometric triple grid diagram $D$ determines a Lagrangian cap $L(D)$ in $(X,\omega)$ for the three disjoint Legendrian links $\Lambda_1,\Lambda_2,\Lambda_3$ determined by the three induced grid diagrams.

    Furthermore, if each $\Lambda_i$ admits a Lagrangian filling in $(B^4,\omega_{std})$, then these fillings can be glued to the Lagrangian cap $L(D)$ to produce an embedded Lagrangian surface.
\end{theorem}

\subsection{Fillability obstructions}

A triple grid diagram $D$ determines a Lagrangian cap $L(D)$ in $(CP^2 \setminus 3 B^4,\omega)$ for a triple $\Lambda_{\alpha \beta} \cup \Lambda_{\beta \gamma} \cup \Lambda_{\gamma \alpha}$ of Legendrian links.  If each Legendrian link admits a Lagrangian filling in $(B^4,\omega_{std})$, then these fillings can be glued to the Lagrangian cap to obtain an embedded, Lagrangian surface in $\CP^2$.  Suppose that each Legendrian link admits a Lagrangian filling by disks.  The Euler characteristic and orientability of the resulting closed surface $L(D)$ is determined by the abstract, Tait-colored cubic graph $\Gamma$.  

The torus $\T^2$ is the only orientable surface that admits a Lagrangian embedding in $\CP^2$.  For nonorientable surfaces, the real projective plane (realized as the set of real points in $\CP^2$) is Lagrangian, but Shevchishin \cite{Shevchishin} and Nemirovski \cite{Nemirovski} showed that the Klein bottle does not admit a Lagrangian embedding.  Furthermore, the combined work of Givental \cite{Givental}, Audin \cite{Audin}, and Dai, Ho, and Li \cite{DHL} establishes that $\#^k \RP^2$ admits a Lagrangian embedding into $\CP^2$ if and only if $k \equiv 2 \text{ (mod 4)}$ and $k \neq 2$ or if $k \equiv 1 \text{ (mod 4)}$.

Therefore, we can obstruct the existence of Lagrangian fillings by the following method.  Take an abstract, Tait-colored cubic graph $\Gamma$ that determines a nonorientable surface with nonpositive Euler characteristic either identically 0 or equal to $2,3 \text{ (mod 4)}$.  Construct the moduli of triple grid diagrams modeled on $\Gamma$ and consider a given triple grid diagram $D$.  If at least two of the resulting Legendrian links admit Lagrangian fillings by disks, then the third Legendrian link cannot.

\subsection{Extension to almost-toric fibrations}

A {\it Markov triple} is a solution $(a,b,c)$ to the Diophantine equation
\[a^2 + b^2 + c^2 = 3abc\]
The base example is $(1,1,1)$ and any other Markov triple can be obtained through a sequence of `mutations'.

Given a Markov triple, the toric structure on the weighted projective plane $\CP^2 (a^2,b^2,c^2)$ can be deformed to an almost-toric structure.  Vianna \cite{Vianna}, for example, showed that the central fiber $T_{a,b,c}$ of such an almost-toric fibration is an exotic monotone Lagrangian torus.  This torus bounds three solid tori $S^1 \times \mathbb{D}^2$ corresponding to three compressing slopes $\alpha,\beta,\gamma \in H_1(T^2)$ satisfying the modified cocycle equation
\[a^2 \alpha + b^2 \beta + c^2 \gamma = 0\]
A neighborhood of the union $\cS = H_{\alpha} \cup H_{\beta} \cup H_{\gamma}$ admits a symplectic structure with three strongly concave boundary components.  Each is a lens space with a universally tight contact structure that can be symplectically filled with a rational homology 4-ball.  Although not stated explicitly in \cite{BGLC_2023}, Theorem \ref{thrm:BGLC} can be immediately extended to triple grid diagrams defined with respect to a triple of slopes $(\alpha,\beta,\gamma)$ defined by a Markov triple.  It may therefore be interesting to study fillable triple grid diagrams with respect to all such slope triples and almost-toric fibrations of $\CP^2$.  We remark that in our method for computing the moduli space $\cM(\Gamma)$ can be immediately adapted to compute the moduli spaces with respect to any slope triple.

\subsection{Acknowledgements}

We would like to thank David Gay, Sarah Blackwell and Jason Cantarella for helpful conversations.

\section{Terminology and Definitions}

\begin{definition}
     A \textit{Triple Grid (combinatorial) diagram} $\D(n,b)$ is a grid diagram of size $n$ with additional information. It consists of the following :
    \begin{enumerate}
     \item a grid on the torus $T^2 = \RR^2/\ZZ^2$ consisting of three sets of lines: 
     \begin{enumerate}
         \item $n$ vertical lines $\left\{ x = \frac{k}{n} : 1 \leq k \leq n \right\}$, colored red by convention,
         \item $n$ horizontal lines $\left\{ y = \frac{k}{n} : 1 \leq k \leq n\right\}$, colored blue by convention, and
         \item $n$ diagonal lines (of slope $-1$) $\left\{x + y = \frac{k}{n} : 1 \leq k \leq n \right\}$, colored green by convention.
     \end{enumerate}
     \item $2b$ points in the complement of the $3n$ grid lines, such that in the region between any pair of adjacent lines of the same slope, there are exactly zero or two points.
 \end{enumerate}
    It can be decomposed into three grid diagrams of size $b$ each.
\end{definition}

\begin{definition}
    A Triple Grid diagram of size $b$ is said to be \textit{orientable} if there exists a partition of the $2b$ points into $b$ $X$s and $b$ $O$s such that each vertical line, a horizontal line, and diagonal line has one $X$ and one $O$. 
\end{definition}

\begin{lemma}
    A Triple Grid diagram being orientable is equivalent to the corresponding cubic Tait-colored graph being bipartite. 
\end{lemma}

\begin{proof}
The proof is by definition.
If $\Gamma$ is bipartite, then we can label the vertices in one partition by $X$s and the other by $O$s. As $\Gamma$ is cubic and Tait-colored, every (red, blue, or green) edge joins an $X$ to an $O$. So, in the corresponding Triple Grid diagrams, this corresponds to 2 points on a line (horizontal, vertical, or diagonal respectively) with one being labeled with $X$ and the other with $O$ i.e. being orientable.
Let $\mathcal{D}$ be an orientable Triple grid diagram. Then we can obtain $\Gamma$ by connecting points with colored edges and forgetting the torus. Since each edge connects an $X$ with an $O$, $\Gamma$ is partitioned into $X$s and $O$s. 
\end{proof}

\begin{definition}
    A \textit{grid diagram} comprises a $n \times n$ grid on a torus with the restriction that each row and column between the grid lines has exactly zero dots or two dots. 
\end{definition}

Note that we are allowing empty rows and columns, unlike those in grid homology. This is because empty rows and columns do not change the link in the grid diagram and are needed to satisfy the diagonal condition in the corresponding Triple grid diagram.

\begin{figure}
    \centering
    \includegraphics[width=10cm]{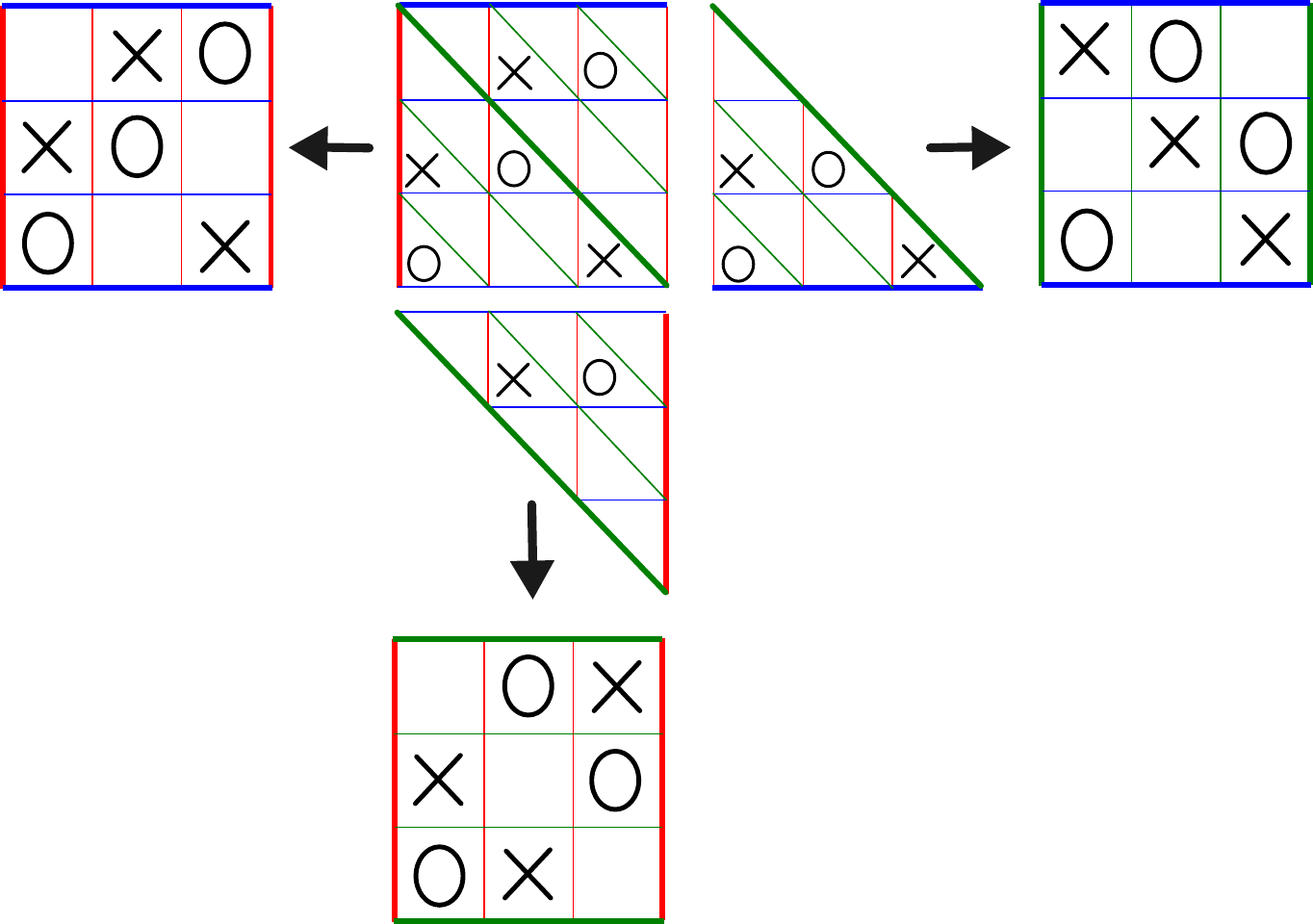}
    \caption{The Triple Grid diagram corresponding to $K_{3,3}$ and the three Grid diagrams it encodes. The upper and lower triangles of the Triple Grid diagram have been repeated to show how to obtain the encoded Grid diagrams.} 
    \label{fig:TripleGrid_3Grid}
\end{figure}

\begin{definition}
A triple grid diagram is {\it fillable} if each of the three grid diagrams it defines, determine front projections of Legendrian links in $(S^3,\xi_{std})$ that admit Lagrangian fillings in $(B^,\omega_{std})$.
\end{definition}

In order to determine whether a grid diagram satisfies the conditions for being a Triple grid diagram, we assume that it is a Triple grid diagram and proceeds to verify that each of the encoded grid diagrams satisfies the relevant restrictions.

This method of searching for combinatorial Triple Grid diagrams of orientable grid number and size $n$ has a complexity of $O((n!)^2)$ since it requires iterating through all possible grid diagrams. As such, we were only able to enumerate all possible Triple grid diagrams up to a size of $n = 9$. Moreover, the output Triple Grids have symmetry coming from translations, swapping of grids, and swapping of colors. This resulted in a lot of repetition in the output. Further, trying to narrow down the output to fillable Triple Grid diagrams was difficult as the output was the order of hundreds of diagrams.

In light of these challenges, we propose an alternative approach for investigating Triple grid diagrams. Specifically, we fix a cubic Tait-colored graph $\Gamma$ and consider all maps of the vertex set $V(\Gamma)$ into $\T^2$. This approach not only provides a more efficient means of studying the moduli space of Triple grid diagrams but also allows for the definition of Triple Grid moves. We are also able to search fillable Triple Grid diagrams more systematically.  

\begin{definition}  
A \textit{cubic (abstract) graph} $\Gamma = (V, E)$ is a graph such that each vertex has degree 3 i.e. 3 edges. A \textit{cubic graph} is said to be \textit{Tait-colored} if we can assign a color (red, blue, or green) to each edge of the graph in such a way that at each vertex, the 3 edges incident to that vertex is colored with exactly three distinct colors. 
\end{definition}

We can define Triple Grid graphs as immersions of Tait-colored cubic graphs into the central torus of $\mathbb{CP}^2$. As we are looking at these graphs from the viewpoint of surfaces, we need them to be well-defined. Therefore, we can not allow double points and overlaps on the central torus. Note that since each point on the central torus has coordinates, this is equivalent to a Geometric Triple Grid diagram. 

\begin{definition}
    Let $\phi : \Gamma \longrightarrow \mathbb{T}^2 = \mathbb{R}^2/\mathbb{Z}^2 $
    such that 
    \begin{itemize}
        \item[1.] $\phi(v_i) = \phi(v_j)$ if and only if $i=j$. 
        \begin{itemize}
            \item[2. a)] $x_i = x_j$ if and only if $ e_{(i,j)} \in E_{red}$. 
            \item[2. b)] $y_i = y_j $ if and only if $e_{(i,j)} \in E_{blue}$.
            \item[2. c)]   $ x_i + y_i = x_j + y_j $ if and only if $e_{(i,j)} \in E_{green}$.
        \end{itemize}
    \end{itemize}   
    Then $\phi(\Gamma)$ defines a Triple Grid diagram.
\end{definition}

\begin{figure}
    \centering
    \def\svgwidth{0.3\textwidth}
    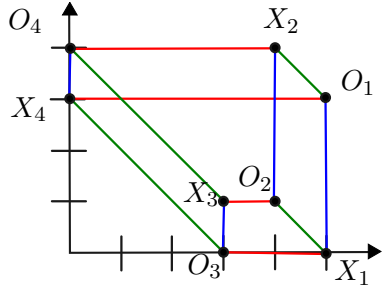
    \caption{Geometric Triple Grid diagram of Torus with $b=4$.} 
    \label{fig:Geometric_t4}
\end{figure}
\begin{figure}
    \centering
    \def\svgwidth{0.3\textwidth}
    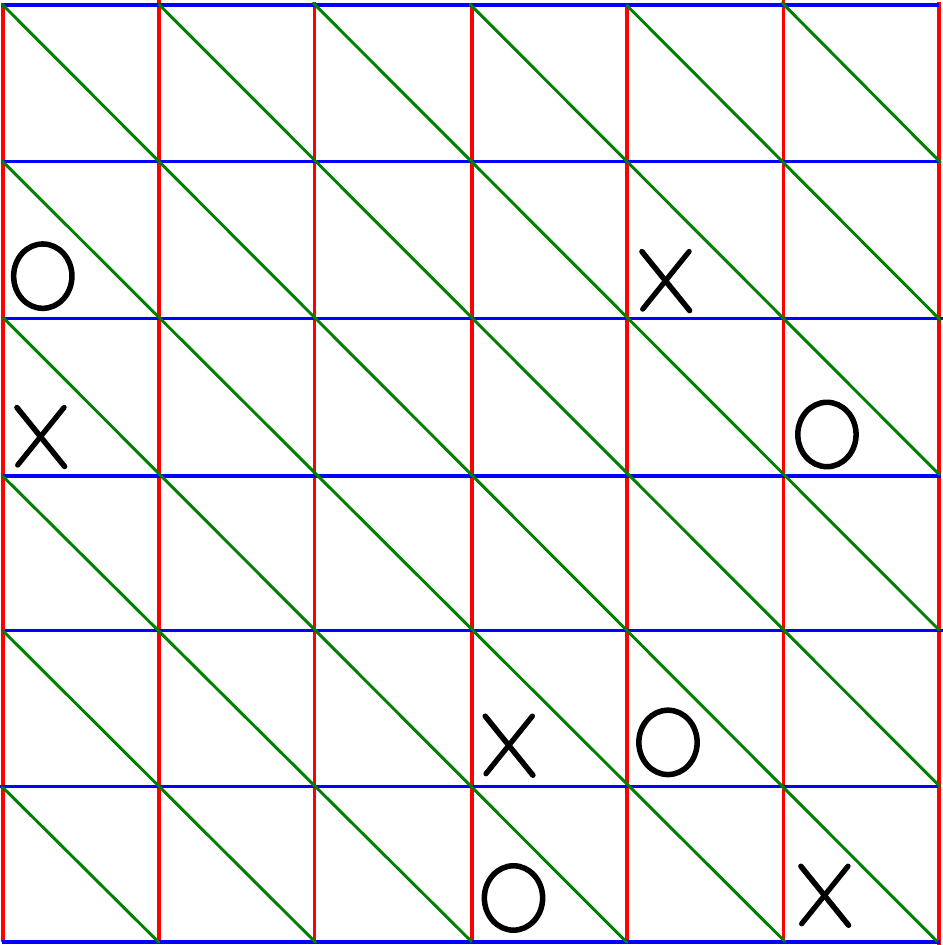
    \caption{(Combinatorial) Triple Grid diagram of Torus with $b=4$.} 
    \label{fig:Geometric_t4}
\end{figure}

\begin{definition}
The \textit{Geometric Triple Grid diagram} $\Theta$ consists of $2b$ points(vertices/dots) in $\mathbb{T}^2 = \mathbb{R}^2/\mathbb{Z}^2$. These points $(x_i,y_i)_1^{2b}$ satisfy the following conditions:
\begin{itemize}
    \item For each line $x = c$, either there are either exactly two points on it or none.
    \item For each $y = c $, either there are either exactly two points on it or none. 
    \item For each line $x+y = c $, either there are either exactly two points on it or none.
\end{itemize}

The Triple Grid can be decomposed into three Grid Diagrams by thinking of a third coordinate $z = -x -y$. Then looking at $(x,y)$,$(y,z)$, and $(z,x)$ gives us the Grid diagrams making up the Triple Grid diagram.
\end{definition}

\begin{definition}
    The \textit{Geometric Grid diagram} consists of $2b$ points(vertices/dots) in $\mathbb{T}^2 = \mathbb{R}^2/\mathbb{Z}^2$. These points $(x_i,y_i)_1^{2b}$ satisfy the following conditions:
\begin{itemize}
    \item For each line $x = c$, either there are either exactly two points on it or none.
    \item For each $y = c $, either there are either exactly two points on it or none. 
\end{itemize}
\end{definition}

\section{Moduli Space of Geometric Triple Grids}

Let $\Gamma = (V,E)$ be a cubic Tait-colored graph.  Define
\[I(\Gamma) := \left\{ \phi: V(\Gamma) \rightarrow \T^2 \right\} \cong (\T^2)^{2b}\]
to be the set of maps of the vertex set $V(\Gamma)$ into $\T^2$.  We can then realize the moduli space of triple grid diagrams with 1-skeleton $\Gamma$ as a subspace of $I(\Gamma)$.  In particular,

\begin{align*}
    \mgamma := \left\{ \begin{tabular}{c|cl} 
         \multirow{3}{*}{$\phi : V(\Gamma) \longrightarrow \mathbb{T}^2 $} 
         & $x_i = x_j$ & if and only if $ e_{(i,j)} \in E_{red} $, \\
    &  $y_i = y_j $ & if and only if $e_{(i,j)} \in E_{blue}$, \\
    & $ x_i + y_i = x_j + y_j $& if and only if $e_{(i,j)} \in E_{green}  $ 
    \end{tabular} \right\} \\   
\end{align*}

The main result of this section is that this moduli space is nonempty if the vertex set is sufficiently large.

\begin{theorem}
\label{thrm:moduli-dimension}
    Let $\Gamma$ be a Tait-colored cubic graph with $2b$ vertices.  The moduli space $\mgamma$ of geometric triple grid diagrams with graph-type $\Gamma$ is a smooth manifold of dimension greater than or equal to $b$.
\end{theorem}

To prove this theorem, we adopt a  more `linear' perspective and alternately describe $I(\Gamma)$ and $\mgamma$ in terms of $\T^2$-valued cochains on $\Gamma$.  We first identify $I(\Gamma)$ with a subspace of $C^*(\Gamma,\T^2)$.  The slope restrictions on the edges can be expressed in terms of a linear subspace $S \subset C^1(\Gamma,\T^2)$ and the nondegeneracy conditions on the vertices can be expressed in terms of a union of hyperplanes $D \subset C^0(\Gamma,\T^2)$.  Then $\mgamma$ corresponds to the subset of $I(\Gamma)$ contained in $S$ but in the complement of $D$.

\subsection{The space $\II(\Gamma)$}

\begin{lemma}
\label{lemma:gamma-immersions}
There are identifications
\[\II(\Gamma) = C^0(\Gamma,\T^2) = \ker d^* \bigoplus \Ima d^* \subset C^0(\Gamma,\T^2) \times C^1(\Gamma,\T^2) \]

\end{lemma}

\begin{proof}

    Let $\phi \in \II(\Gamma)$ be a map. It assigns a point in $\T^2$ to each vertex
    \begin{align*}
        \phi(v_i) &= \Vec{x_i} \in \T^2 
    \end{align*}
    On the other hand, a 0-cochain in $C^0(\Gamma,\T^2)$ is by definition a homomorphism $\Tilde{X}:C_0(\Gamma,\T^2) \longrightarrow \T^2$.  Let $x_0 \in \T^2$ be a fixed basepoint.  Then there is a basis for $C_0(\Gamma,\T^2)$ consisting of the 0-chains that assigns $x_0$ to each vertex.  Consequently, we can identify the map $\phi$ with the 0-cochain that sends the $v_i$ 0-chain to $\phi(v_i) \in \T^2$.  
    
    The second identification is given by the first isomorphism theorem for linear maps.  Explicitly, we can identify $\phi$ with a pair $(X,W)$ such that 
    \begin{align*}
        X:\{v_i\} &\longrightarrow \T^2 & W:\{e_{ij}\} &\longrightarrow \T^2 \\
        X(v_i) &= \Vec{x_i} & 
        W(e_{ij}) &= \Vec{x_j} - \Vec{x_i} 
    \end{align*}
\end{proof}




The kernel of $d^*$ is precisely the constant 0-cochains and can be identified with $\T^2$.  Geomtrically, this summand correponds to translations of $\phi$ in $\T^2$.  Passing to homology, we have the diagram

\[\begin{tikzcd}
	{C^0 (\Gamma, \mathbb{T}^2)} & {C^1 (\Gamma, \mathbb{T}^2)} & 0 \\
	{} & {H^1(\Gamma,\mathbb{T}^2)}
	\arrow["{d^*}", two heads, from=1-1, to=1-2]
	\arrow[from=1-2, to=1-3]
	\arrow["{i^*}", from=1-2, to=2-2]
\end{tikzcd}\]
and exactness implies that $\Ima d^* = \ker i^* $.

\begin{corollary} 
\label{cor:gamma-immersions}
There is an identification
    \[\II(\Gamma) = \T^2 \oplus \ker i^*\]
\end{corollary}
Denote $I_0(\Gamma) \coloneqq \Ima d^* \cong \ker i^*$.

\subsection{Slope Restrictions}

The slope restrictions on a triple grid diagram can be reinterpreted in terms of a linear subspace in $C^1(\Gamma,\T^2)$.

The restriction on the coordinates that $x_i= x_j$ if $ e_{(i,j)} \in E_{red} $ can be rephrased as $W(e_{ij}) = \vec{x_j} - \vec{x_i} \in (\mathbb{S}^1 \times 0) \subset \T^2$, if $ e_{(i,j)} \in E_{red} $. This is a hyperplane in $C^1(\Gamma,\T^2)$. 
Repeating this for each edge, the slope restraints arising from the coloring correspond to restricting to a subset $S$ of $C^1(\Gamma,\T^2)$ that is the common intersection of multiple hyperplanes.

\[
    S := \left\{ 
    \begin{tabular}{c|ccccc} 
         \multirow{3}{*}
         {$ W : C_1 (\Gamma,\T^2) \longrightarrow \T^2$}
    & $W(e_{ij})$ &$= \vec{x_j} - \vec{x_i}$ & $\in (\mathbb{S}^1 \times 0)$ & if $ e_{(i,j)} \in E_{red} $ \\
    & $W(e_{ij})$ &$= \vec{x_j} - \vec{x_i}$ & $\in (0 \times \mathbb{S}^1)$ & if $ e_{(i,j)} \in E_{blue} $\\
    & $W(e_{ij})$ &$= \vec{x_j} - \vec{x_i}$ & $\in \Delta = \{ (x,x) \in \T^2 \}$ & if $ e_{(i,j)} \in E_{green} $
    \end{tabular} 
    \right\}
\]

Let $\mathcal{M}_{0}(\Gamma)$ denote {\it based} Geometric Triple Grid diagrams, where a designated vertex $v_1$ is mapped to a distinguished basepoint $(0,0)$ in $\T^2$:

\[
    \mathcal{M}_0(\Gamma):= \left\{ \begin{tabular}{c|cl} 
         \multirow{4}{*}{$\phi : \Gamma \longrightarrow \mathbb{T}^2 $} 
    & $x_i = x_j$ &  if $ e_{(i,j)} \in E_{red} $, \\
    &  $y_i = y_j $ &  if $e_{(i,j)} \in E_{blue}$, \\
    & $ x_i + y_i = x_j + y_j $& if $e_{(i,j)} \in E_{green}  $,\\
    & $\phi(v_1) = (0,0)$ & 
    \end{tabular} \right\} \\   
\]
Then
\begin{lemma}
\label{lemma:m0-dimension}
There is an inclusion
\[ \mgamma \subset  (\ker d^* )\oplus (\Ima d^* \cap S) \cong  \T^2 \oplus  I_0(\Gamma) \]
Moreover, the dimension of moduli of based Geometric Triple Grid diagrams satisfies
\[\dim \mathcal{M}_0(\Gamma) \geq b - 2 \]   
\end{lemma}

\begin{proof}

   The inclusion is true by construction.

   The dimension of $I(\Gamma)$ is given by
        \[\dim (\II(\Gamma)) = \dim ( C^0 (\Gamma,\T^2) ) = |V| \times \dim (\T^2) = 2b \times 2 = 4b\]
    and $\dim \Ima d^* = \dim (\II(\Gamma)) - 2$.  The dimension of $S$ is exactly half the dimension of $C^1(\Gamma,\T^2)$, since each 1-cochain is constrained to lie in a 1-dimensional subspace of $\T^2$.  Therefore
        \[\dim S = \frac{1}{2} \dim C^1(\Gamma,\T^2) = \frac{1}{2} (|E|\times 2) = |E| = 3b\]
    Combining this, we have
\[\dim (\Ima d^* \cap S) \geq \dim \Ima d^* + \dim S - \dim C^1(\Gamma,\T^2) = 4b - 2 + 3b -7b = b-2 \]
   So, $\dim (\Ima d^* \cap S) = \dim \mathcal{M}_0(\Gamma) \geq b- 2 $.

\end{proof}

\subsection{Degenerate Immersions}

Similarly, the nondegeneracy conditions on a triple grid diagram can be reinterpreted in terms of the complement of a hyperplane arrangement in $C^0(\Gamma,\T^2)$.

To ensure that $x_i = x_j$ only if $e_{(i,j)} \in E_{red}$, we need to remove the cases when $x_i = x_k$ for $e_{(i,k)} \notin E_{red}$. 
Let $\Tilde{X} \in C^0(\Gamma, \T^2)$. Then 

\begin{align*}
    \Tilde{X} : C_0 (\Gamma,\T^2) &\longrightarrow \T^2\\
                 v_i       &\longmapsto (x_i,y_i)
\end{align*}
Define 
\[ R_{i,k} := \left\{ \Tilde{X}: x_i = x_k \right\}\]
It is a codimension-1 hyperplane in $C^0(\Gamma,\T^2)$.  Let $\cR_{i,k} = R_{i,k} \oplus C^1(\Gamma,\T^2)$ and 
\[\cR = \bigcup_{\substack{i,k \\ e_{i,k} \notin E_{red}} } \cR_{i,k}\]
Define $B_{i,k},\cB$ and $G_{i,k},\cG$ similarly for the blue and green edges, respectively.  Then $D = \cR \cup \cB \cup \cG$ is a union of hyperplanes in $C^*(\Gamma,\T^2)$ and a triple grid diagram must lie in the complement of $D$.

This gives us the following proposition.

\begin{proposition}
\label{prop:main-moduli-description}

    \[\mgamma =  \big ( (\ker d^* )\oplus (\Ima d^* \cap S) \big )\setminus D = \big ( (\T^2)\oplus \mathcal{M}_0(\Gamma) \big )\setminus D  \]
    
\end{proposition}   

Since $D$ has codimension 1, removing it does not affect the dimension.  Then Theorem \ref{thrm:moduli-dimension} follows from the above proposition and Lemma \ref{lemma:m0-dimension}.

\subsection{Examples}

\subsubsection{}
\label{ex:b1}
Let us consider the case when $b=1$. Then the only cubic Tait colored graph possible is the theta graph. Since the two vertices are connected by edges of all three colors, the immersion needs to connect two points with lines having slope $0$, $1$ and $\infty$. The only immersion possible is degenerate. 

\subsubsection{}
\label{ex:b2}
 There is only one cubic Tait colored graph having $b=2$, which is shown in Figure \ref{fig:b2_graph}.
 This graph represents $\RP2$. From Proposition 4.5,  $\dim \mgamma \geq 2 $ and $\dim \mathcal{M}_0(\Gamma) \geq 0$.
 Due to the slope restritions, there is only one Geometric Triple Grid diagram upto translation, shown in Figure \ref{fig:b2_GTGD}.
 Hence these inequalities are strict for $\RP2$.   

\begin{figure}[H]
\centering
\begin{subfigure}{.25\textwidth}
  \centering
  \includegraphics[width = 0.8 \textwidth]{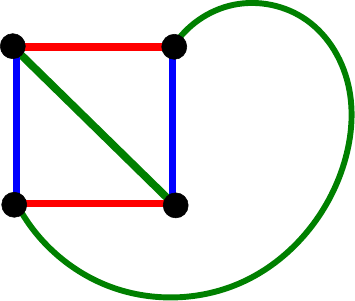}
  \caption{$\Gamma$ with $b=2$}
  \label{fig:b2_graph}
\end{subfigure}%
\begin{subfigure}{.35\textwidth}
  \centering
 \def\svgwidth{0.9\textwidth}
\begingroup%
  \makeatletter%
  \providecommand\color[2][]{%
    \errmessage{(Inkscape) Color is used for the text in Inkscape, but the package 'color.sty' is not loaded}%
    \renewcommand\color[2][]{}%
  }%
  \providecommand\transparent[1]{%
    \errmessage{(Inkscape) Transparency is used (non-zero) for the text in Inkscape, but the package 'transparent.sty' is not loaded}%
    \renewcommand\transparent[1]{}%
  }%
  \providecommand\rotatebox[2]{#2}%
  \newcommand*\fsize{\dimexpr\f@size pt\relax}%
  \newcommand*\lineheight[1]{\fontsize{\fsize}{#1\fsize}\selectfont}%
  \ifx\svgwidth\undefined%
    \setlength{\unitlength}{359.70273512bp}%
    \ifx\svgscale\undefined%
      \relax%
    \else%
      \setlength{\unitlength}{\unitlength * \real{\svgscale}}%
    \fi%
  \else%
    \setlength{\unitlength}{\svgwidth}%
  \fi%
  \global\let\svgwidth\undefined%
  \global\let\svgscale\undefined%
  \makeatother%
  \begin{picture}(1,0.96746058)%
    \lineheight{1}%
    \setlength\tabcolsep{0pt}%
    \put(0,0){\includegraphics[width=\unitlength,page=1]{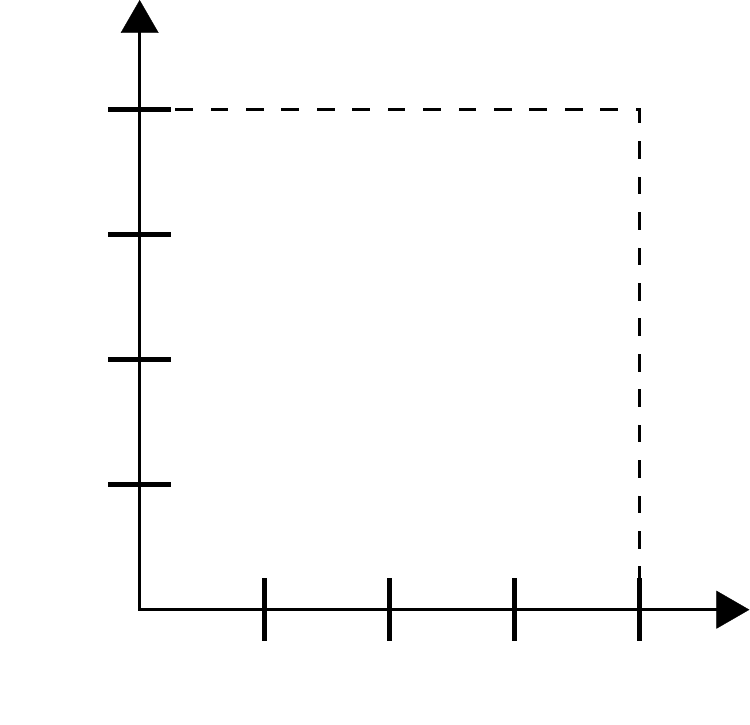}}%
    \put(-0.00028649,0.01818436){\color[rgb]{0,0,0}\makebox(0,0)[lt]{\lineheight{1.25}\smash{\begin{tabular}[t]{l}$0$\\\end{tabular}}}}%
    \put(0.41821385,0.01222706){\color[rgb]{0,0,0}\makebox(0,0)[lt]{\lineheight{1.25}\smash{\begin{tabular}[t]{l}$0.5$\\\end{tabular}}}}%
    \put(0.79650238,0.0092484){\color[rgb]{0,0,0}\makebox(0,0)[lt]{\lineheight{1.25}\smash{\begin{tabular}[t]{l}$1$\\\end{tabular}}}}%
    \put(-0.00475442,0.46796055){\color[rgb]{0,0,0}\makebox(0,0)[lt]{\lineheight{1.25}\smash{\begin{tabular}[t]{l}$0.5$\\\end{tabular}}}}%
    \put(-0.00475442,0.80156935){\color[rgb]{0,0,0}\makebox(0,0)[lt]{\lineheight{1.25}\smash{\begin{tabular}[t]{l}$1$\\\end{tabular}}}}%
    \put(0,0){\includegraphics[width=\unitlength,page=2]{b2_geometric_triple_grid_rp2.pdf}}%
  \end{picture}%
\endgroup%

  \caption{Geometric Triple Grid Diagram}
  \label{fig:b2_GTGD}
\end{subfigure}
\caption{Example with $b=2$ representing $\RP2$}
\label{fig:b2_example}
\end{figure}

 \subsubsection{}
 Consider $\Gamma$ to be the graph $K_{3,3}$, having $b=3$. This symmetric graph has only Tait-coloring possible depicted in Figure \ref{fig:b3_graph}. Due to Theorem \ref{thm:main_thm},we know that $\dim \mathcal{M}(K_{3,3}) \geq 3 $ and $\dim \mathcal{M}_0(K_{3,3}) \geq 1$. By calculation, we see that $\mathcal{M}'(K_{3,3}) \cong \T^2$ represented in Figure \ref{fig:b3_moduli}, and so $\dim \mathcal{M}(K_{3,3}) = 4$.

 \begin{figure}[H]
\centering
\begin{subfigure}{.3\textwidth}
  \centering
  \includegraphics[width = 0.8 \textwidth]{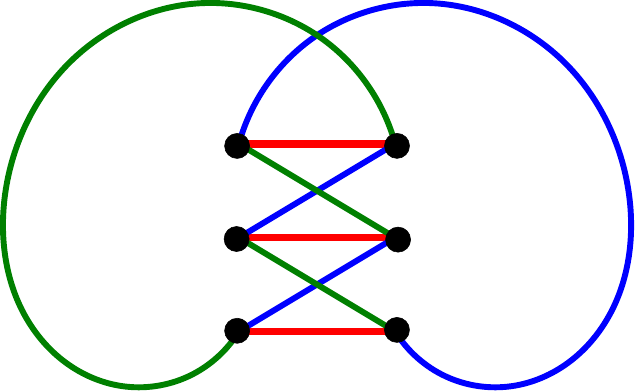}
  \caption{$\Gamma$ with $b=3$ }
  \label{fig:b3_graph}
\end{subfigure}%
\begin{subfigure}{.3\textwidth}
  \centering
 \def\svgwidth{0.9\textwidth}
\begingroup%
  \makeatletter%
  \providecommand\color[2][]{%
    \errmessage{(Inkscape) Color is used for the text in Inkscape, but the package 'color.sty' is not loaded}%
    \renewcommand\color[2][]{}%
  }%
  \providecommand\transparent[1]{%
    \errmessage{(Inkscape) Transparency is used (non-zero) for the text in Inkscape, but the package 'transparent.sty' is not loaded}%
    \renewcommand\transparent[1]{}%
  }%
  \providecommand\rotatebox[2]{#2}%
  \newcommand*\fsize{\dimexpr\f@size pt\relax}%
  \newcommand*\lineheight[1]{\fontsize{\fsize}{#1\fsize}\selectfont}%
  \ifx\svgwidth\undefined%
    \setlength{\unitlength}{373.83744336bp}%
    \ifx\svgscale\undefined%
      \relax%
    \else%
      \setlength{\unitlength}{\unitlength * \real{\svgscale}}%
    \fi%
  \else%
    \setlength{\unitlength}{\svgwidth}%
  \fi%
  \global\let\svgwidth\undefined%
  \global\let\svgscale\undefined%
  \makeatother%
  \begin{picture}(1,0.93088112)%
    \lineheight{1}%
    \setlength\tabcolsep{0pt}%
    \put(0,0){\includegraphics[width=\unitlength,page=1]{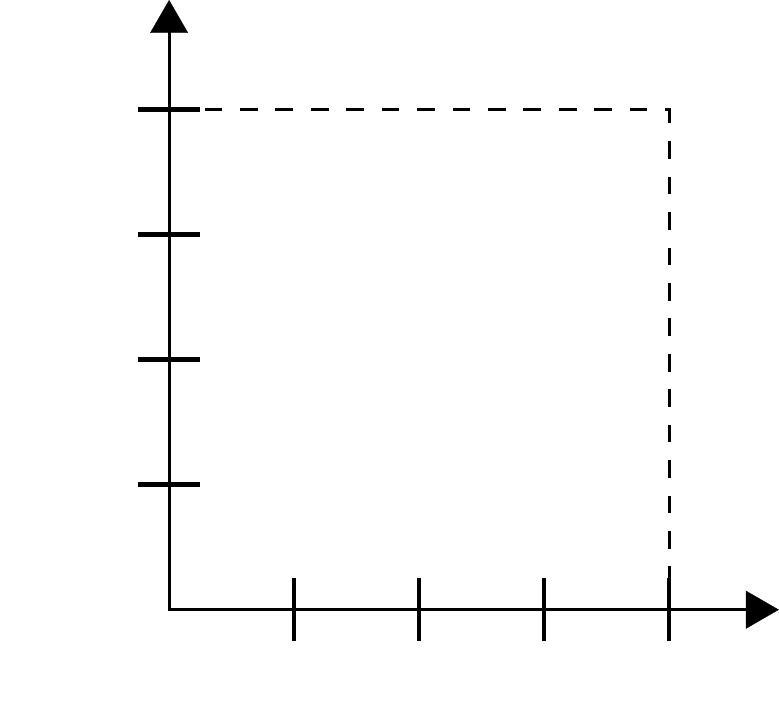}}%
    \put(0.02606998,0.01749681){\color[rgb]{0,0,0}\makebox(0,0)[lt]{\lineheight{1.25}\smash{\begin{tabular}[t]{l}$0$\\\end{tabular}}}}%
    \put(0.48320146,0.00889872){\color[rgb]{0,0,0}\makebox(0,0)[lt]{\lineheight{1.25}\smash{\begin{tabular}[t]{l}$0.5$\\\end{tabular}}}}%
    \put(0.80419659,0.00889872){\color[rgb]{0,0,0}\makebox(0,0)[lt]{\lineheight{1.25}\smash{\begin{tabular}[t]{l}$1$\\\end{tabular}}}}%
    \put(-0.00457466,0.45026707){\color[rgb]{0,0,0}\makebox(0,0)[lt]{\lineheight{1.25}\smash{\begin{tabular}[t]{l}$0.5$\\\end{tabular}}}}%
    \put(0.02177098,0.7712622){\color[rgb]{0,0,0}\makebox(0,0)[lt]{\lineheight{1.25}\smash{\begin{tabular}[t]{l}$1$\\\end{tabular}}}}%
    \put(0,0){\includegraphics[width=\unitlength,page=2]{b3_geometric_triple_grid_k33.pdf}}%
  \end{picture}%
\endgroup%

  \caption{Geometric Triple Grid }
  \label{fig:b3_GTGD}
\end{subfigure}
\begin{subfigure}{.3\textwidth}
  \centering
 \def\svgwidth{0.9\textwidth}
\begingroup%
  \makeatletter%
  \providecommand\color[2][]{%
    \errmessage{(Inkscape) Color is used for the text in Inkscape, but the package 'color.sty' is not loaded}%
    \renewcommand\color[2][]{}%
  }%
  \providecommand\transparent[1]{%
    \errmessage{(Inkscape) Transparency is used (non-zero) for the text in Inkscape, but the package 'transparent.sty' is not loaded}%
    \renewcommand\transparent[1]{}%
  }%
  \providecommand\rotatebox[2]{#2}%
  \newcommand*\fsize{\dimexpr\f@size pt\relax}%
  \newcommand*\lineheight[1]{\fontsize{\fsize}{#1\fsize}\selectfont}%
  \ifx\svgwidth\undefined%
    \setlength{\unitlength}{430.14967568bp}%
    \ifx\svgscale\undefined%
      \relax%
    \else%
      \setlength{\unitlength}{\unitlength * \real{\svgscale}}%
    \fi%
  \else%
    \setlength{\unitlength}{\svgwidth}%
  \fi%
  \global\let\svgwidth\undefined%
  \global\let\svgscale\undefined%
  \makeatother%
  \begin{picture}(1,0.88851382)%
    \lineheight{1}%
    \setlength\tabcolsep{0pt}%
    \put(0,0){\includegraphics[width=\unitlength,page=1]{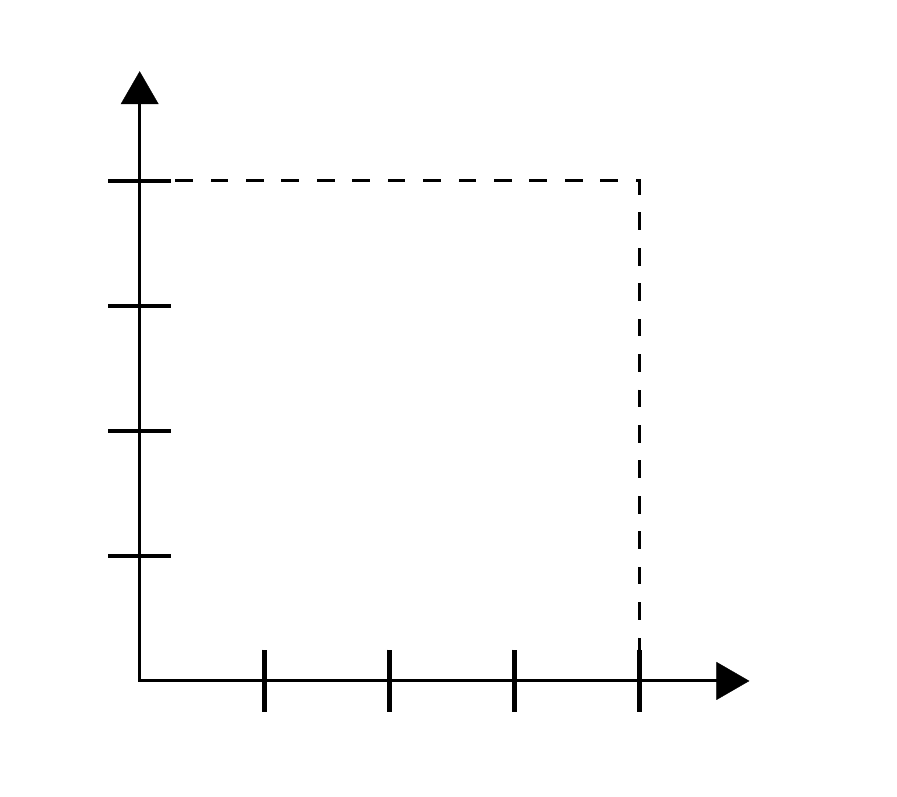}}%
    \put(-0.00023957,0.01520625){\color[rgb]{0,0,0}\makebox(0,0)[lt]{\lineheight{1.25}\smash{\begin{tabular}[t]{l}$0$\\\end{tabular}}}}%
    \put(0.34972167,0.01022459){\color[rgb]{0,0,0}\makebox(0,0)[lt]{\lineheight{1.25}\smash{\begin{tabular}[t]{l}$0.5$\\\end{tabular}}}}%
    \put(0.66605673,0.00773376){\color[rgb]{0,0,0}\makebox(0,0)[lt]{\lineheight{1.25}\smash{\begin{tabular}[t]{l}$1$\\\end{tabular}}}}%
    \put(-0.00397578,0.39132121){\color[rgb]{0,0,0}\makebox(0,0)[lt]{\lineheight{1.25}\smash{\begin{tabular}[t]{l}$0.5$\\\end{tabular}}}}%
    \put(-0.00397578,0.67029387){\color[rgb]{0,0,0}\makebox(0,0)[lt]{\lineheight{1.25}\smash{\begin{tabular}[t]{l}$1$\\\end{tabular}}}}%
    \put(0,0){\includegraphics[width=\unitlength,page=2]{Moduli_space_b3_k33.pdf}}%
    \put(0.8572401,0.12332131){\color[rgb]{0,0,0}\makebox(0,0)[lt]{\lineheight{1.25}\smash{\begin{tabular}[t]{l}$p_1$\end{tabular}}}}%
    \put(0.00413162,0.8518882){\color[rgb]{0,0,0}\makebox(0,0)[lt]{\lineheight{1.25}\smash{\begin{tabular}[t]{l}$p_2$\end{tabular}}}}%
  \end{picture}%
\endgroup%

  \caption{Moduli Space }
  \label{fig:b3_moduli}
\end{subfigure}
\caption{Example Graph $K_{3,3}$ with $b=3$ representing }
\label{fig:b3}
\end{figure}

The pink line in the center is a plane in $D$, which denotes degenerate immersions of $K_{3,3}$, along with the lines $p_1 = 0$ and $p_2 = 0$ The yellow region and the cyan blue region correspond to two regions in $ \mathcal{M}'(K_{3,3})$. Each region maps to Geometric Triple Grid diagrams which differ by minor perturbations and therefore correspond to the same Triple Grid diagram. Going from one Triple Grid diagram to another requires passing through a hyperplane in $D$. Therefore, Triple Grid moves can be defined as moving through a degenerate immersion. These moves do not change the underlying graph $\Gamma$ and therefore so not change the surface in the bridge trisection.

 \subsubsection{}

For $b=4$, consider $\Gamma$ to be the cubic Tait-colored graph in Figure \ref{fig:b4_graph} which corresponds to $\T^2$. Due to Theorem \ref{thm:main_thm}, we know that $\dim \mathcal{M}(\Gamma) \geq 4 $ and therefore $\mathcal{M}_0(\Gamma) \geq 2$.

\begin{figure}[H]
   \centering
   \begin{subfigure}{.3\textwidth}
        \includegraphics[width = 0.8\textwidth]{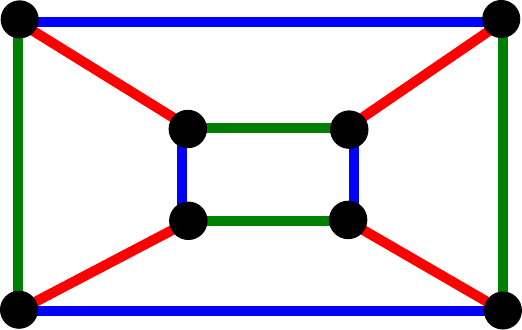}
    \caption{$\Gamma$ with $b=4$}
    \label{fig:b4_graph}
    \end{subfigure}
   \begin{subfigure}{.5\textwidth}   
      \centering
        \def\svgwidth{0.6\textwidth}
\begingroup%
  \makeatletter%
  \providecommand\color[2][]{%
    \errmessage{(Inkscape) Color is used for the text in Inkscape, but the package 'color.sty' is not loaded}%
    \renewcommand\color[2][]{}%
  }%
  \providecommand\transparent[1]{%
    \errmessage{(Inkscape) Transparency is used (non-zero) for the text in Inkscape, but the package 'transparent.sty' is not loaded}%
    \renewcommand\transparent[1]{}%
  }%
  \providecommand\rotatebox[2]{#2}%
  \newcommand*\fsize{\dimexpr\f@size pt\relax}%
  \newcommand*\lineheight[1]{\fontsize{\fsize}{#1\fsize}\selectfont}%
  \ifx\svgwidth\undefined%
    \setlength{\unitlength}{539.68687256bp}%
    \ifx\svgscale\undefined%
      \relax%
    \else%
      \setlength{\unitlength}{\unitlength * \real{\svgscale}}%
    \fi%
  \else%
    \setlength{\unitlength}{\svgwidth}%
  \fi%
  \global\let\svgwidth\undefined%
  \global\let\svgscale\undefined%
  \makeatother%
  \begin{picture}(1,1.1748251)%
    \lineheight{1}%
    \setlength\tabcolsep{0pt}%
    \put(0,0){\includegraphics[width=\unitlength,page=1]{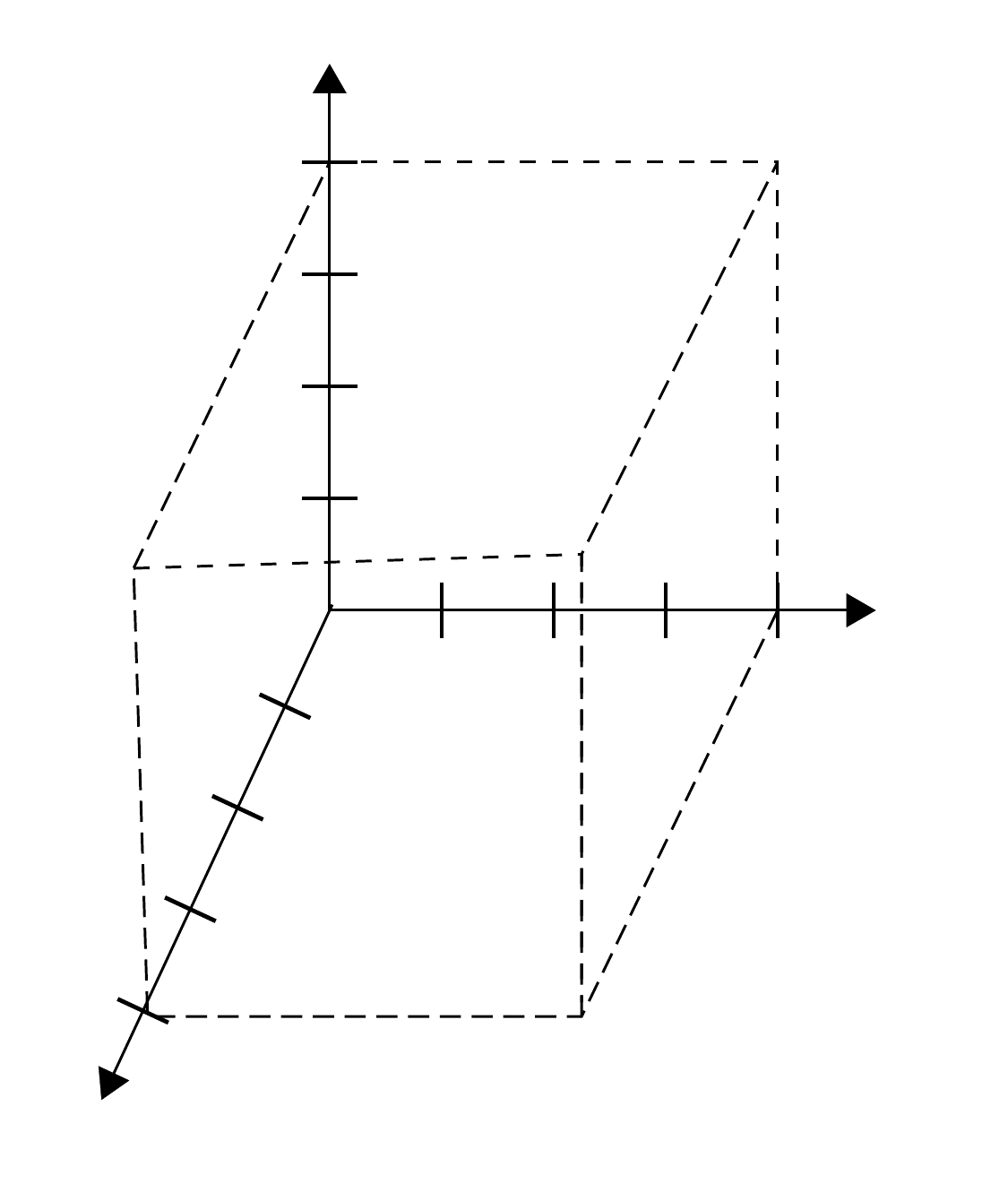}}%
    \put(0.88621527,0.56493941){\color[rgb]{0,0,0}\makebox(0,0)[lt]{\lineheight{1.25}\smash{\begin{tabular}[t]{l}$p_1$\end{tabular}}}}%
    \put(0.20625736,1.14563318){\color[rgb]{0,0,0}\makebox(0,0)[lt]{\lineheight{1.25}\smash{\begin{tabular}[t]{l}$p_2$\end{tabular}}}}%
    \put(0,0){\includegraphics[width=\unitlength,page=2]{Moduli_space_b4_torus.pdf}}%
    \put(-0.00318935,0.0090615){\color[rgb]{0,0,0}\makebox(0,0)[lt]{\lineheight{1.25}\smash{\begin{tabular}[t]{l}$p_3$\end{tabular}}}}%
    \put(0,0){\includegraphics[width=\unitlength,page=3]{Moduli_space_b4_torus.pdf}}%
  \end{picture}%
\endgroup%

    \caption{Moduli Space}
    \label{fig:b4_moduli}
    \end{subfigure}
\caption{Example with $b=4$ representing $\T^2$}
\label{fig:b4_graph_grid}
\end{figure}

By calculation, we get that $\mathcal{M}_0(\Gamma) \cong \T^3$ as shown in Figure \ref{fig:b4_moduli}. The moduli space is divided in to 6 regions by 6 degenerate planes, which correspond to $p_1 = 0$,$p_2 = 0$,$p_3 = 0$, $p_1 = p_2$, $p_2 = p_3$, $p_3 = p_1$ in this parameterized representation. We can see Geometric Triple Grid Diagrams from each of the 6 differently colored regions, represented in Figure \ref{fig:b4_GTGD}. These give us the corresponding Combinatorial Triple Grid Diagrams, as shown in Figure \ref{fig:b4_CTGD}. Therefore, by fixing the underlying graph $\Gamma$, we are able to find all the equivalent Triple Grid Diagrams. We can also see many Triple Grid moves, like the one shown in Figure \ref{fig:TripleGridMove}.

Note that the degenerate planes $p_1 = p_2$, $p_2 = p_3$ and $p_3 = p_1$ intersect at the line $p_1 = p_2 = p_3$ represented in pink in the figure. Therefore, we can see that all 6 regions share this line in their boundary.

\begin{figure}
\centering
\begin{subfigure}[b]{0.75\textwidth}
  \def\svgwidth{\textwidth}
        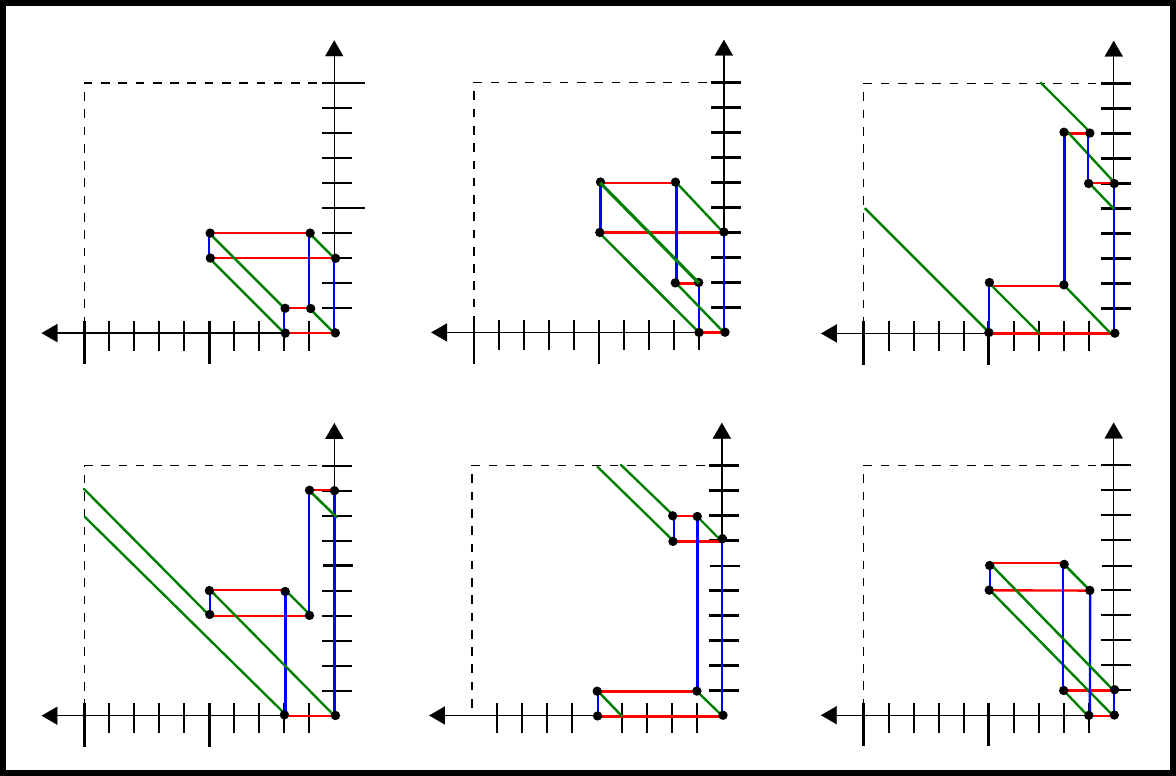
        \caption{Six Geometric Triple Grid diagrams corresponding to the six regions in the moduli space for the $b=4$ example in Figure \ref{fig:b4_graph_grid}}
        \label{fig:b4_GTGD}
\end{subfigure}

\end{figure}

\begin{figure}[H]
    \centering
    \includegraphics[width = 0.4\textwidth]{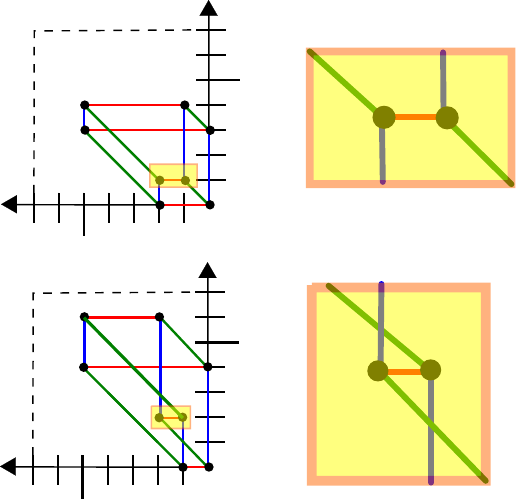}
    \caption{Example of Triple Grid Move}
    \label{fig:TripleGridMove}
\end{figure}

\section{Computation} 

Cochains with $\T^2$-coefficients are computationally awkward.  Instead, we can use the fact that $\T^2 = \RR^2/\ZZ^2$ and work with $\RR^2$-valued cochains instead.  The moduli $\cM(\Gamma)$ of triple grid diagrams can be modeled as a {\it subset} of $C^0(\Gamma,\T^2)$, but a suitable cover of $\cM(\Gamma)$ can be realized as a collection of affine subspaces in $C^1(\Gamma,\RR^2)$.  In particular, these are solution sets to systems of linear equations and can therefore be computed in polynomial time.

From the short exact sequence
\[0 \rightarrow \ZZ^2 \rightarrow \RR^2 \rightarrow \T^2 \rightarrow 0\]
we obtain a short exact sequence of cochain complexes
\[\begin{tikzcd}
	& 0 & 0 & 0 \\
	0 & {C^1(\Gamma, \mathbb{Z}^2)} & {C^1(\Gamma, \mathbb{R}^2)} & {C^1(\Gamma, \mathbb{T}^2)} & 0 \\
	0 & {C^0(\Gamma, \mathbb{Z}^2)} & {C^0(\Gamma, \mathbb{R}^2)} & {C^0(\Gamma, \mathbb{T}^2)} & 0
	\arrow[from=2-2, to=1-2]
	\arrow[from=2-3, to=1-3]
	\arrow[from=2-4, to=1-4]
	\arrow["{d^*}", two heads, from=3-2, to=2-2]
	\arrow["{d*}", two heads, from=3-3, to=2-3]
	\arrow["{d^*}", two heads, from=3-4, to=2-4]
	\arrow[from=2-1, to=2-2]
	\arrow[from=3-1, to=3-2]
	\arrow["{d_s^*}", hook, from=2-2, to=2-3]
	\arrow["{d_s^*}", hook, from=3-2, to=3-3]
	\arrow["{d_r^*}", two heads, from=2-3, to=2-4]
	\arrow["{d_r^*}", two heads, from=3-3, to=3-4]
	\arrow[from=2-4, to=2-5]
	\arrow[from=3-4, to=3-5]
\end{tikzcd}\]

We identify $I(\Gamma)/\T^2$ with the subspace of coexact cochains in $C^1(\Gamma,\T^2)$.  Via the short exact sequence of chain complexes, we can lift every 1-cochain in $C^1(\Gamma,\T^2)$ to a 1-cochain in $C^1(\Gamma,\RR^2)$.  Since the diagram is commutative, it follows that every exact 1-cochain $\x$ in $C^1(\Gamma,\T^2)$ lifts to an exact 1-cochain $\xtilde$ in $C^1(\Gamma,\RR^2)$.  The slope conditions $S \subset C^1(\Gamma,\T^2)$ also lifts to a linear subspace $\widetilde{S}$ in $C^1(\Gamma,\RR^2)$ that contains the origin.

However, it does not follow that $\x$ satisfies the slope condition $S$ only if the lift $\xtilde$ lies in the subspace $\widetilde{S}$.  The preimage of $I(\Gamma)$ in $C^1(\Gamma,\RR^2)$ is strictly larger than the subspace of coexact 1-cochains in $C^1(\Gamma,\RR^2)$.  From the short exact sequence of complexes, we obtain a surjection
\[C^1(\Gamma,\ZZ^2) + \text{Im}(d^0_{\RR}) \longrightarrow \text{Im}(d^0_{\T^2}) \subset C^1(\Gamma,\T^2)\]
In fact, we can strengthen this to a surjection
\[\text{Coker}(d^0_{\ZZ}) \oplus \text{Im}(d^0_{\RR}) \rightarrow \text{Im}(d^0_{\T^2})\]
since $\text{Im}(d^0_{\ZZ})$ injects into $\text{Im}(d^0_{\RR})$.  Finally, we can identify $\text{Coker}(d^0_{\ZZ})$ with $H^1(\Gamma,\ZZ^2)$.  Summarizing this discussion, we have the following proposition.

\begin{proposition}
    Let $\x \in C^1(\Gamma,\T^2)$ be an exact 1-cocycle.  Then $\x$ satisfies the slope condition $S$ if and only if there exists a lift $\xtilde \in \text{Im}(d^0_{\RR})$ and a class $\rho \in H^1(\Gamma,\ZZ^2) \cong \text{Coker}(d^0_{\ZZ}) \cong \text{Im}(d^0_{\ZZ})^{\perp}$ such that
    \[\xtilde + \rho \in \widetilde{S}\]
    Moreover, the lift $\xtilde$ is unique up to an element $\ytilde \in \widetilde{S} \cap \text{Im}(d^0_{\ZZ})$ and $\rho$ is unique up to an element of $\text{Im}(d^0_{\ZZ})^{\perp} \cap \widetilde{S}$.
\end{proposition}

This proposition implies that we can decompose $\cM(\Gamma)$ into components $\cM_{\rho}(\Gamma)$, indexed by elements $\rho \in H^1(\Gamma,\ZZ^2)$, and (modulo translation in $\T^2$) each of these components is covered by an affine subspace of $C^1(\Gamma,\RR^2)$.

\begin{proposition}
    Let $\rho \in H^1(\Gamma,\ZZ^2) \cong \text{Im}(d^0_{\ZZ})^{\perp}$.  There is a surjective map
\[ \T^2 \oplus \left( \left( \rho + \text{Im}(d^0_{\RR}) \right) \cap \widetilde{S} \right) \rightarrow \cM_{\rho}(\Gamma) \]
that restricts to a covering map
\[\left( \T^2 \oplus \left( \left( \rho + \text{Im}(d^0_{\RR}) \right) \cap \widetilde{S} \right) \right) \setminus \widetilde{D} \rightarrow \cM_{\rho}(\Gamma)\]
where $\widetilde{D}$ denotes the preimage in $C^*(\Gamma,\RR^2)$ of the degeneracy locus $D \subset C^*(\Gamma,\T^2)$.
\end{proposition}

To summarize, we can efficiently compute $\cM(\Gamma)/\T^2$ as follows:
\begin{enumerate}
    \item The maps $d^0_{\ZZ}, d^0_{\RR}$ are essentially the adjacency matrix for the abstract graph $\Gamma$.  
    \item Choose an element 
    \[\rho \in \text{Coker}(d^0_{\ZZ}) \cong \text{Im}(d^0_{\ZZ})^{\perp}\]
    and view it as an element of $C^1(\Gamma,\RR^2)$.
    \item The slope restrictions $\widetilde{S}$ are determined by the Tait-coloring of $\Gamma$ and the assignment of slopes in $\RR^2$ to edge colors.
    \item We then have
    \[\cM_{\rho}(\Gamma)/\T^2 = \left( \rho + \text{Im}(d^*_{\RR}) \right) \cap \widetilde{S}\]
    \item Letting $\rho$ vary in $H^1(\Gamma,\ZZ^2) \cong \text{Im}(d^*_{\ZZ})^{\perp}$ yields the entire space $\cM(\Gamma)/\T^2$.
\end{enumerate}

In general, this method may overcount triple grid diagrams but this indeterminacy is exactly determined by the discrete subset $\widetilde{S} \cap C^1(\Gamma,\ZZ^2)$.

\printbibliography[title={References Cited}]

@article{MZ-bridge_2017,
   title={Bridge trisections of knotted surfaces in $S^4$.},
   volume={369},
   ISSN={1088-6850},
   url={http://dx.doi.org/10.1090/tran/6934},
   DOI={10.1090/tran/6934},
   number={10},
   journal={Transactions of the American Mathematical Society},
   publisher={American Mathematical Society (AMS)},
   author={Meier, Jeffrey and Zupan, Alexander},
   year={2017},
   month=may, pages={7343–7386} }

@misc{BGLC_2023,
      title={Constructing Lagrangians from triple grid diagrams}, 
      author={Sarah Blackwell and David T. Gay and Peter Lambert-Cole},
      year={2023},
      eprint={2306.16404},
      archivePrefix={arXiv},
      primaryClass={math.GT}
}

@article{DHL,
  title={Nonorientable Lagrangian surfaces in rational {$4$}-manifolds},
  author={Dai, Bo and Ho, Chung-I and Li, Tian-Jun},
  journal={Algebraic \& Geometric Topology},
  volume={19},
  number={6},
  pages={2837--2854},
  year={2019},
  publisher={Mathematical Sciences Publishers}
}

@article{Givental,
  title={Lagrangian imbeddings of surfaces and unfolded Whitney umbrella},
  author={Givental, Aleksandr Borisovich},
  journal={Functional Analysis and Its Applications},
  volume={20},
  number={3},
  pages={197--203},
  year={1986},
  publisher={Springer}
}

@article{Nemirovski,
  title={Homology class of a Lagrangian Klein bottle},
  author={Nemirovski, Stefan Yurievich},
  journal={Izvestiya: Mathematics},
  volume={73},
  number={4},
  pages={689--698},
  year={2009},
  publisher={Turpion Ltd}
}

@article{Shevchishin,
  title={Lagrangian embeddings of the Klein bottle and combinatorial properties of mapping class groups},
  author={Shevchishin, Vsevolod V.},
  journal={Izvestiya: Mathematics},
  volume={73},
  number={4},
  pages = {797--859},
  year={2009},
  publisher={IOP Publishing}
}

@article{Audin,
  title={Quelques remarques sur les surfaces lagrangiennes de Givental},
  author={Audin, Mich{\`e}le},
  journal={Journal of Geometry and Physics},
  volume={7},
  number={4},
  pages={583--598},
  year={1990},
  publisher={Elsevier}
}

@article {Vianna,
    AUTHOR = {Vianna, Renato Ferreira de Velloso},
     TITLE = {Infinitely many exotic monotone {L}agrangian tori in
              {$\Bbb{CP}^2$}},
   JOURNAL = {J. Topol.},
  FJOURNAL = {Journal of Topology},
    VOLUME = {9},
      YEAR = {2016},
    NUMBER = {2},
     PAGES = {535--551},
      ISSN = {1753-8416,1753-8424},
   MRCLASS = {53D12 (53D37 53D42)},
  MRNUMBER = {3509972},
MRREVIEWER = {Stefan\ Nemirovski},
       DOI = {10.1112/jtopol/jtw002},
       URL = {https://doi.org/10.1112/jtopol/jtw002},
}

\end{document}